\documentclass[12pt]{paper}
\usepackage{amssymb,amsrefs,amsthm,amsfonts,amsmath,latexsym}
\usepackage{color}
\usepackage{graphicx}
\usepackage[utf8]{inputenc}
\usepackage[english]{babel}
\usepackage{cite}
\usepackage{marginnote}
\usepackage{enumitem}

\newtheorem{theorem}{Theorem}
\newtheorem{lemma}[theorem]{Lemma}

\newtheorem{remark}[theorem]{Remark}

\newcommand\R{{\mathbb R}}

\renewcommand\P{{\mathbb P}}

\newcommand\E{{\mathbb E}}

\newcommand{\Cov}{{\rm Cov}}
\newcommand{\Var}{{\rm Var}}

\providecommand{\keywords}[1]
{
  \small	
  \textbf{\textit{Keywords---}} #1
}

\begin{document}
\title{Necessary and sufficient conditions for the finiteness  of  the second moment of the measure of level sets }
\author{J-M. Aza\"{i}s\thanks{IMT, Universit\'{e} de Toulouse, 
Toulouse, France. Email: jean-marc.azais@math.univ-toulouse.fr}
\qquad
J. R. Le\'{o}n\thanks{IMERL, Facultad de Inegenier\'{\i}a
Universidad de la Rep\'{u}blica, Montevideo, Uruguay. Email: rlramos\,@fing.edu.uy and Universidad Central de Venezuela. Escuela de Matem\'atica.}}
\maketitle
\begin{abstract}
For a smooth  vectorial stationary Gaussian  random field $X:\Omega\times\R^d\to\R^d$, we give necessary and sufficient conditions to have  a finite second moment 
for the number of roots of $X(t)-u$. The results are obtained by using a method of proof inspired on the one obtained  by D. Geman for stationary Gaussian processes long time ago.  Afterwards the same method is applied to the number of critical points of a scalar random field and also to the level set of a vectorial process $X:\Omega\times\R^D\to\R^d$ with $D>d$.
\end{abstract}
\\
\textup{2000} \textit{Mathematics Subject Classification}: \textup{60G15, 60G60}

\section{Introduction}
In the 1940s three articles with apparently different orientations appeared in mathematical literature. Firstly was  Mark Kac's paper \cite{Kac:Kac} ``On the average number of real roots of a random algebraic equation" and secondly two papers written by S.O. Rice \cite{Rice1}, \cite{Rice2}``Mathematical analysis of random noise I and II".  In the work of Kac and in the second  of Rice the zeros of  Gaussian random functions were studied. In particular they established with precision a formula, known today as the Kac-Rice formula, which allows to compute the expectation of the number of  zeros (or crossings by   any level) of a Gaussian random function. In spite of the apparently separated that seem the works, M. Kac in the review of the article affirms that ``All these results (of Rice) can also be derived using the methods introduced by the reviewer (M. Kac)". 

After these two works an intense research activity has been developed. In particular, the interest in these subjects had a great impulse after the appearance of the book written by H. Cramer and M. R. Leadbetter \cite{C:L}. In this work, there is not only a general demonstration of the Kac-Rice formula for the number of crossings of a Gaussian processes, but also formulas for the factorial moments of this last random variable. An important fact to notice is that in the book  a sufficient condition for the second moment of the number of crossings of zero to be finite is established. Then a little time later D. Geman  in \cite{Ge} showed that this condition was also necessary.  This  condition is now known as `` the Geman condition". This result has been extended to  any level at \cite{Kratz:Leon}.

The theme gained a new impulse when appear in the eighties two books, the first one written by R. Adler \cite{Ad1} ``The geometry of random fields"  and the second one a Lecture Notes \cite{W1} written by M. Wschebor ``Surfaces al\'eatoires. Mesure g\'eom\'etrique des ensembles de niveau". Both books focus their study on crossings or geometric invariants of the level sets, for random fields having a multidimensional domain and taking scalar or vector values. The problems studied by Cramer \& Leadbetter were extended to this new context. In particular we must point out the Adler \& Hasofer's article \cite {Ad:Ha} in which conditions are established so that the number of stationary points for a Gaussian  field of $ X: \R ^ 2 \to \R $ have a second moment. It is important to observe that studying the stationary points of a scalar field leads to study the zeros of its gradient, which is a vector field. 

The twenty-first century saw two books appear \cite{Ad:Tay} and \cite{AW} that gave a new impetus to the subject. New fields of application of the formulas appeared in the literature and the area has become a large domain of research.  We can point out for instance the applications to the number of roots of random polynomial systems (algebraic or trigonometric) and also to the volume of nodal sets when the systems are rectangular \cite{Peccati}. Also Kac-Rice formulas are a basic tool to study the sets of zeros of random waves and  it has been much effort to prove or disprove Berry's conjectures \cite{berry1}, see \cite{Peccati} and the references therein. A field of applications where the formulas have been very useful is in random sea modeling, the Lund's School of probability has been very active in these matters, see for instance the paper \cite{Lindgren} and the references therein. In addition, the processes to which the crossings are studied can have their domain in a manifold of finite dimension  see \cite{let1}. A very interesting case of this last situation is the article \cite{Au:Ben} where the domain of the random filed is the sphere in  large dimension.

In the present paper we obtain necessary and sufficient conditions to have a finite second moment for the number of roots of $X(t)-u$, for a stationary, mean zero Gaussian field $X:\Omega\times\R^d\to\R^d$. The proofs of the main results are  rather simple using the case $d=1$ as inspiration.  Our results can be extended to the number of critical points of a stationary mean zero scalar Gaussian field. We must note that recently in \cite{estr:four} a sufficient condition for the critical points of a scalar field has finite second moment was given, however our method is rather different. Finally let us point us that as a bonus our method of proof allows  obtaining a very simple result for the volume of level sets for Gaussian fields $X:\Omega\times\R^D\to\R^d$ with $D>d.$ Under condition of stationarity and diferentiability, the second moment is always finite.

Suppose that we have a way to check easily  that the measure of the level set of a Gaussian field has finite second moment. Then it is ready to obtain an It\^o-Wiener expansion for this functional. Two  consequences of this representation are important to remark: firstly the asymptotic variance of the level functional can be computed and also the speed of the divergence of this quantity can be estimated, secondly the fourth moment theorem can be used to obtain diverse CLT. This has been done  some time ago in \cite{KL} and more recently in a lot of papers. We can cite by instance the article \cite{Peccati} where one can also consult some recent references.

The organisation of the
 paper is the following: in Section \ref{Geman1} we revisit the results of \cite{Kratz:Leon} in dimension 1. Section \ref{dd} studies  the number of points of  levels sets  for  a  random field $X: \R^d \to \R^d$, $d>1$. The subsection \ref{criti} is devoted to the study  of the number of critical points  of a random field $X: \R^d \to \R$. Section  \ref{Dd} studies the   measure of levels sets for  a  random field $X: \R^D\to \R^d$, $D>d$.  The proof of the different lemmas are given in the appendix. 
\label{dd}

\section{ Real valued process on the line, Geman's condition}\label{Geman1}
The results of this section are contained in the paper \cite{Kratz:Leon}.  However, we present a new proof as an introduction to the next section.

\noindent Consider a process $X: \R\to\R$ and assume
\begin{itemize}
\item It is Gaussian stationary, normalized this is:
$$\E(X(0))=0;\quad \Var(X(t))=1.$$
This last point is without loss of generality. 
\item The second spectral moment  $\lambda_2$  is positive and finite.
\end{itemize}
Let $N_u([0,T]):=\#\{t\in[0,T]:\, X(t)=u\}$ for a given level $u\in \R$. 
Moreover we define the covariance  $$r(\tau)=\E[X(0)X(\tau)].$$
 And let set  $$\sigma^2(\tau):=\Var(X'(0)|\; X(0)=X(\tau)=0)=\lambda_2-\frac{(r'(\tau))^2}{1-r^2(\tau)}.$$ 
In what follows $(Const)$ will denote a generic  positive constant, its value  can change from one occurence to another.\\ The relation  $x \leq (Const) y, y \leq (Const)x$ is denoted 
$ x\asymp y$. \medskip

The object of this section is to prove the next theorem:
\begin{theorem} \label{t:1}
The following statements are equivalent
\begin{enumerate}[label=(\alph*)]
\item $\E(N_u([0,T])^2)$ is finite for some $u$ and $ T$.
\item $\E(N_u([0,T])^2)$ is finite for all $u$ and all finite $T$.
\item  The intergral $\int \frac{\sigma^2(\tau)}{\tau}d\tau$  converges at zero.
\end{enumerate}
\end{theorem}

{\bf Remark: } Integrating by parts  in (c)  we get the classical Geman's condition by using the following lemma, whose proof (as well as the proofs of all lemmas) is referred to the appendix.
\begin{lemma}\label{integrales}
There  is equivalence  between the convergence at zero of the two following integrals
$$
 \int \frac{\lambda_2+r''(\tau)}{\tau}d\tau \quad \mbox { and } \quad \int \frac{\sigma^2(\tau)}{\tau}d\tau.$$
\end{lemma}

Before the proof of the theorem we need some notation and two lemmas.
\begin{lemma} \label{resultados}  For $\tau$ sufficiently small  we  set the following definitions and we have the following   relations.

\begin{enumerate}[label=(\alph*)]
\item $\mu_{1,\tau,u}:=\E(X(\tau)|\, X(0)=X(\tau)=u)=\frac{r'(\tau)u}{1+r(\tau)}$.
\item $\mu_{2,\tau,u}:=\E(X(0)|\, X(0)=X(\tau)=u)=-\mu_{1,\tau,u}.$
\item Recall that   $\sigma^2(\tau)=\Var(X(0)|\, X(0),\, X(\tau))=\lambda_2-\frac{(r'(\tau))^2}{1-r^2(\tau)}.$ 
\item $\det(\Var(X(0);X(\tau))=1-r^2(\tau)\asymp \tau^2$
\item   if the fourth spectral  moment  $\lambda_4$  satisfies  $ \lambda_2^2 < \lambda_4 \leq  + \infty$, then \\ $\frac{|\mu_{1,\tau,u}|}{\sigma(\tau)}\leq (Const) u$ .
\end{enumerate}
\end{lemma}

\begin{lemma} \label{acotacion}Assume that $|m_1|, |m_2|\le K$ for some constant $K$ and that $(Y_1;Y_2)\stackrel{\mathcal L}= N\big((m_1;m_2),\begin{pmatrix}1&\rho\\\rho&1\end{pmatrix}\big)$. Then 
$\E|Y_1Y_2|\asymp 1.$ Where the  two constants implied in the symbol $\asymp$ depend on $K$.
\end{lemma}

\begin{proof}[Proof  of the Theorem]
First we have to consider  the particular case $\lambda_4 = \lambda_2^2$. This corresponds to the Sine-Cosine process: $X(t) = \xi_1 \sin(wt) + \xi_2 \cos(wt)$ where $ \xi_1,\xi_2$ are independent standard normals. In this case  a direct calculation shows that (a)-(c) hold true. 

We consider now the other cases assuming that $\lambda_2^2<\lambda_4$. We start from (c): we assume that 
$$ \int_0^T \frac{\sigma^2(\tau)}{\tau}d\tau <+\infty  \quad\mbox{ with }T \mbox{ sufficiently small.} 
$$
The expectation of the number of crossings is finite because the second spectral moment  is, see \cite{C:L}. Thus it is enough to work with the second factorial moment.  The Kac-Rice formula for this quantity \cite{C:L} writes
\begin{align}
 \E &(N_u([0,T])(N_u([0,T])-1))= \notag \\
\
&\frac1\pi \int_0^T(T-\tau)\E[|X'(0)||X'(\tau)|\,|X(0)=X(\tau)=u)\frac{e^{-\frac {u^2}{1+r}}}{\sqrt{1-r^2}} d\tau. \notag
\\
\label{KR}
& \leq (Const) \int_0^T    \E[| \frac{X'(0)}{\sigma(\tau)}| | \frac{X'(\tau)}{\sigma(\tau)}|\, \Big|X(0)=X(\tau)=u) \frac{ \sigma^2(\tau) }{\tau^2} d\tau,
\end{align}
using Lemma \ref{resultados}  (d).   
By Lemma  \ref{resultados} (e), $\frac{X'(0)}{\sigma(\tau}$ and $\frac{X'(\tau)}{\sigma(\tau)}$ have a bounded  conditional mean, then applying now Lemma \ref{acotacion}:
\begin{equation} \label{e:jma1}
\E (N_u([0,T])(N_u([0,T])-1) \leq (Const) \int_0^T   \frac{ \sigma^2(\tau) }{\tau^2} d\tau.
\end{equation}
This give the finiteness of the second moment  form $T$ sufficiently small. By  the Minkowsky inequality  it is also the case for every $T$ giving (b). 

 In the other direction we start from (a)  with $u=0$ and $T $ sufficiently small (which is weaker than (b)) and we prove (c) .
 
 Again  we can consider the second factorial moment  and apply  the Kac-Rice  formula to get that \\
\\
  $ \displaystyle
\E (N_u([0,T])(N_u([0,T])-1)$ $$\geq (Const) \int_0^{T/2}    \E\Big(| \frac{X'(0)}{\sigma(\tau)}| | \frac{X'(\tau)}{\sigma(\tau)}|\, \Big|X(0)=X(\tau)=u \Big)      \frac{ \sigma^2(\tau) }{\tau^2} d\tau.
$$ 
It suffices to apply Lemma \ref{acotacion} in the  other direction. 
\end{proof}
\begin{remark}  We can also obtain \eqref{e:jma1} with an explicit constant by use of the Cauchy-Schwarz inequality. 
\end{remark}
%
%
\section{Random fields $\R^d \to \R^d$, $d>1$}
\label{dd}
%
 \subsection{Position of the problem}
 
 Let us consider a random field $X: \R^d\to\R^d$. We assume  $(H_1)$:
\begin{itemize}
\item The field is Gaussian and stationary and has a continuous derivative.
\item The distribution of $X(0)$ (respectively $X'(0)$) is non degenerate (N.D.).
\end{itemize}
By a rescaling in space we can assume without loss of generality that
$$\E[X(t)]=0\quad \mbox{ and } \quad \Var(X(t))=I_d,$$ where $\Var$ denotes  for us the variance-covariance matrix.
We keep the notation $\Cov$  for the  matrix
$$
\Cov(X,Y) := \E\Big( \big(X-\E(X)\big) \big(Y-\E(Y)\big)^\top\Big).
$$
 We also define the following  additional hypothesis
%
 \begin{equation} \label{H2}
 \tag{$H_2$}\mbox{The coordinates }X_i \mbox{ of }X \mbox{  are independent and isotrope} 
\end{equation}

We define 
\begin{align*}
\sigma^2_{i,\lambda}(r) &:=    \Var\big(X'_{i\lambda}\big |\, X(0), X(\lambda r)\big),\\
\sigma^2_{max} (r) &:= \max_{i =1,\ldots,d}\max_{\lambda\in \mathbb S^{d-1}}   \sigma^2_{i,\lambda}(r), 
 \end{align*}
 where $X'_{i\lambda}$ denotes the derivative of $X_i$ in the direction $\lambda \in  \mathbb S^{d-1}$.

 \subsection{Zero level}\label{zerolevel}
We set $N(0,S)$ the number of roots of the field $X(\cdot)$  on some compact set $S$. The following result is new as well as all that follows.
 
 \begin{theorem}\label{t:jma12} Under $(H1)$,
  if $$
  \int \frac{\sigma^2_{max} (r)}{r} dr  \mbox{ converges at } 0,
 $$
 then for all compact $S \subset  \R^d $ : $\E\big((N(0,S))^2\big)$ is finite. 
 \end{theorem}
 The proof of the Theorem uses the following lemma.
   \begin{lemma}  \label{l:gos}
   Let $T,(Z_n)_n$ be in the same Gaussian space. Assume that $Z_n\to Z$ a.s. or in probability or in $\mathbb L^2(\Omega)$ and the random variable $Z$ is (N.D.). Then
 \begin{align*}
 \forall z,\, &\E(T|\, Z_n=z)\to\E(T|\, Z=z), \\
 &\Var(T|Z_n) \to \Var(T|Z).
 \end{align*}
 \end{lemma}

 \begin{proof}[Proof of Theorem \ref{t:jma12}]

 Set
 $$\mathcal C = \{ X(0) = X(t) =0\}, $$
and  let $E_\mathcal C $ denotes the expectation conditional to $\mathcal C$.

We consider the following quantity
 \begin{equation} \label{atu}
 \mathcal A(t,0)=\E_\mathcal C\big(|\det X'(0)\det X'(t)|).
\end{equation}
By applying the Cauchy-Schwarz inequality  and by symmetry  of the roles of $0$ and $t$:
$$
\mathcal A(t,0) \leq \E_\mathcal C \big( \det((X'(0))^\top X'(0) )\big),
$$
We   define the Jacobian the matrix $X'(0)$ by 
$X'_{ij}(0) =  \frac{\partial X_i}{ \partial t_j} $. 

We perform a change of  basis  so that $ t = r e_1=|t| e_1$ where $e_1$ is  the first vector  of the new  basis. We denote by $ \bar X$  the expression of $X$ in this basis. Let $ \bar X'_{:j} $ denote the $j$th column of $ \bar X'$. Using Gram  representation of  the semidefinite  positive matrix $M=(M_{ij})$, we know that  
\begin{equation} \label{zaza}
  \det(M) \leq  M_{1,1} \ldots M_{d,d}.
  \end{equation}
  This gives 
  \begin{equation} \label{som}
  \mathcal A(t,0) \leq   \E_\mathcal C \big( \| \bar X'_{:1}\|^2 \ldots  \| \bar X'_{:d}\|^2\big)  
  = \sum_{ 1\leq i_1,\ldots,i_d \leq d} \E_\mathcal C( (X'_{i_1,1})^2 \ldots  (X'_{i_d,d})^2).
  \end{equation}
  Because the conditional  expectation is contractive, for $j>1$,
  \begin{equation} \label{sc1}
   \E_\mathcal C \big(  (\bar X'_{i_j,j})^2 \big) \leq  \E\big( ( \bar X'_{i_j,j})^2 \big)  \leq (Const).
   \end{equation}
  In addition 
   \begin{equation} \label{sc2}
   \E_\mathcal C \big(  (\bar X'_{i_1,1})^2 \big) \leq  \sigma^2_{max} (r)  .
   \end{equation}

   If we consider a term of \eqref{som}, we can apply Cauchy-Schwarz inequality to get that it is bounded by
   $$
  \big( E_\mathcal C( (X'_{i_1,1})^4\big)^{1/2} \big( E_\mathcal C( (X'_{i_2,2})^4\ldots  (X'_{i_d,d})^4)\big)^{1/2} .
 $$ 
 Using \eqref{sc1} and  \eqref{sc2}, we get that this term is bounded by 
 $$
   (Const) \sigma^2_{max}(r). 
  $$
  As a consequence  we get the same bound for the whole sum.  \medskip
  
 We now study  the joint density  $$p_{X(0),X(t)} (0,0) = (Const) \big( \det \Var(X(0),X(t))\big) ^{-\frac 1 2} .$$ 
  Using the fact that a determinant is invariant  by adding to  some row (or column) a linear combination of the others rows (or columns) we get
  $$ 
  \det (\Var(X(0),X(t))) = \det (\Var(X(0),X(t) -X(0))).
  $$
  Using Lemma \ref{l:gos}. 
 \begin{equation}\label{density}
   p_{X(0),X(t)} (0,0)   \simeq  (Const)  r^{-d} ( \det \Var(X(0),X_\lambda'(0))\big) ^{-\frac 1 2} 
  \asymp r^{-d} ,
  \end{equation}
  where $\lambda:= t/ \|t\|$.

  We are now able to apply  the Kac-Rice formula see, for example, \cite{AW}, Theorem 6.3. As in the case $d=1$  we can limit our attention  to the second factorial moment. We have
 \begin{align*}
 \E &(N(0,S)(N(0,S)-1) \notag \\
=& \int_{S^2}  \mathcal A(t-s,0) P_{X(s),X(t)} (0,0) ds dt 
\le (Const) |S| \int_{S}    \sigma^2_{max}(t) \|t\|^{-d}  dt, 
\end{align*}
where $|S|$ is the Lebesgue measure of $S$. 
Passing  in polar coordinates, including  $S$ in a centered ball with  radius $a$, we get  that the term above is bounded  by
 $$
(Const) \int _0^a\ r^{d-1} r^{-d}  \sigma^2_{max}(r)  d r = (Const)  \int _0^a  \frac{  \sigma^2_{max}(r)} {r}   d r.
$$ 
  \end{proof}

 \subsection{General level}
 
 In this section we assume $(H_1)$ and $(H_2)$. Note that
$ \sigma^2_{i,\lambda} (r)$ depends no more on $\lambda$.  We denote by 
$ \sigma^2_{i} (r)$ its value.  We have 
 $ \sigma^2_{max} (r) = \max_{i =1,\ldots,d}\sigma^2_{i}(r)$.
 
   Our result is the following

  \begin{theorem}\label{t:jma} Under the hypotheses above \\
 $\bullet$  If $$
  \int \frac{ \sigma^2_{max} (r)}{r} dr  \mbox{ converges at } 0,
 $$
 then for all compact $S \subset  \R^d $  an all $ u \in \R^d $: $\E\big((N(u,S))^2\big)$ is finite. 
 
  $\bullet$    
  If  $\E\big((N(u,S))^2\big)$ is finite  for some $u$ and some compact $S$  with non-empty interior, then 
 $$
  \int \frac{ \sigma^2_{max} (r)}{r} dr  \mbox{ converges at } 0.
 $$
\end{theorem} \medskip

 
 Because of stationarity and isotropy  we have 
$$
\Cov(X_i (s), X_i (t)) =  \rho_i(\|s-t\|^2),
$$
where $\rho_i$ is some function of class $C^2(\mathbb R)$.

 Before the proof of the Theorem we state the following lemmas.  
 
 \begin{lemma} \label{Chichi} Let $\mathcal F$  be a family  of Gaussian distributions for 
 $X,Y$ two $d\times d $ Gaussian matrices. Let $Z$  be the $2 d^2$ vector obtained by the elements of $X,Y$  in any order.

 $(a)$  Suppose that  for all distribution in $\mathcal F$, $\E(Z) \in K_1$ and $\Var (Z) \in K_2$  where $K_1,K_2$ are two compacts sets . 

 Then  there exists $C$ such that :
 $$ 
 \sup_{f \in \mathcal F} \E_f(| \det(X) \det(Y)|) \leq C.$$
The constant $C$  depends only on $K_1$, $K_2$ and $d$.
 
  $(b)$ Suppose in addition that   for every $f \in \mathcal F$,   $$\P\{ \det(X) =0\}=0, \P\{ \det(Y) =0\}=0$$
  then 
  there exists $c$ such that :
 $$ 
 \E (| \det(X) \det(Y)|) \geq c.
 $$
 The  positive constant $c$  depends only on   $K_1$, $K_2$ and $d$.
 \end{lemma}
 For establishing the  next lemma let introduce the following definitions
  $$\sigma^2_i(r)=-2\rho'_i(0)-\frac{4r^2(\rho'_i(r^2))^2}{1-\rho^2_i(r^2)},$$
$$b_i(r)\sigma_i(r)=\big({-2\rho'_i(r^2)-4r^2\rho_i''(r^2)-\frac{4r^2\rho_i(r^2)(\rho'_i(r^2))^2}{1-\rho^2_i(r^2)}}\big).$$

Then we have the following, denoting $\Var_\mathcal C$ the variance-covariance matrix conditional to $\mathcal C$.

\begin{lemma}\label{l:cov}[see \cite{AW} p. 336]
 
$${\small\Var_\mathcal C \big(X'_i(0); X'_i(r e_1)\big)} ={\tiny\begin{bmatrix}\sigma_i(r)&0&\ldots&0&b_i(r)\sigma_i(r)&0&\ldots&0\\
0&-2\rho'_i(0)&\ldots&0&0&0&\ldots&0\\
\ldots&\ldots&\ldots&\ldots&\ldots&\ldots&\ldots&\ldots\\
0&0&\ldots&-2\rho'_i(0)&0&0&\ldots&0\\
b_i(r)\sigma_i(r)&0&\ldots&0&\sigma_i(r)&0&\ldots&0\\
0&0&\ldots&0&0&-2\rho'_i(0)&\ldots&0\\
\ldots&\ldots&\ldots&\ldots&\ldots&\ldots&\ldots&\ldots\\ 
0&0&\ldots&0&0&0&\ldots&-2\rho'_i(0)\end{bmatrix}}.$$   
\end{lemma}
 \begin{proof}[Proof  of the Theorem \ref{t:jma}]
 
 We begin by considering  instead of (\ref{atu}) the quantity 
  \begin{equation} \label{atu}
 \mathcal A(t,u)=\E_\mathcal C\big(|\det X'(0)\det X'(t)|),
\end{equation}
where now
$$\mathcal C=\{X(0)=X(t)=u\}.$$

Because  of isotropy we can assume, without loss of generality,  that  $ t = r e_1$. Because of the independence of each coordinates assumed in \eqref{H2}
$$
\E_\mathcal C(X'_{i,1}(0)) = \E(X'_{i,1}(0) \big |\, X_i(0) = X_i(r e_1) =u_i).
$$

So we have to consider a one dimensional problem as in Section \ref{Geman1}.
 In addition  the spectral measure  of each $X_i $  is invariant 
by isometry   so its projection on the first axis cannot be reduced to  one point (or two taking into account symmetry). As a consequence
Lemma \ref{resultados} $(e)$ holds true  implying that 
 $$
 | \E_\mathcal C(X'_{i,1}(0)  | \leq (Const) u_i \sigma_i(r) .
$$

Let us consider now $E_\mathcal C(X'_{i,j}) =  \E\big(X'_{i,j} \big |\,  X_i(0) = u_i,  \frac{X_i(t) -X_i(0)} {r} =0)$ for $j \neq 1$.
From Lemma \ref{l:gos} 
$$E_\mathcal C(X'_{i,j}) \simeq \E\big(X'_{i,j} \big |\,  X_i(0) = u_i, X'_{i,1}(0) =0\big).
$$
By independence
$$
\E_\mathcal C(X'_{i,j}) \simeq \E\big(X'_{i,j} \big |\,X'_{i,1}(0)  = 0 ) = 0.
$$
Of course we have the same  kind of result for $X'(r e_1)$.

 So, if we divide  the first column of $X'(0)$ and $X'(r e_1)$ by $\sigma_{max}(r)$  to obtain $ \tilde X'(0)$ and $ \tilde X'(r e_1)$,    Lemma   \ref{l:cov}
 implies  that 
 all the terms  of the variance-covariance matrix are bounded, the expectation is bounded. Using lemma \ref{Chichi} we get that 
 \begin{equation}\label{e:majo}
 \mathcal A(t,u) \leq (Const)  \sigma_{max}^2(r).
 \end{equation}
 The end  of  the proof of the first assertion is  similar to the one of Theorem \ref{t:jma12}. \medskip

 
We turn now to the proof   of the second assertion.  
To get the inequality on the other direction we must  carefully apply (b) of Lemma \ref{Chichi}. For this purpose we need to describe the compact sets $K_1$ and $K_2$.

We know, from the proof of the first assertion,  that all the expectations are bounded so $K_1$ is just $[-a,a]^{2d^2\times2d^2}$ for some $a$.
For  the domain $K_2$ of the variance-covariance matrix of   $\tilde X' (0),\tilde X' (t) $ : 
 \begin{itemize} 
 \item First we have an independence between the coordinates $X_i$. Denoting by 
 $X'_{i,:} (t)$ the $i$ row of $X'$ i.e. the gradient of $X_i$ at $t$, then the variables $ \big(X'_{i,:} (0),X'_{i,:} (t) \big), 
 i=1,\ldots,d$ are independent.
 \item If we consider $ \big(X'_{i,:} (0),X'_{i,:} (t) \big)$ for some fixed $i$, we see from lemma \ref{l:cov} that (i) only 
 one variance varies  : $\sigma_i(r)$,  (ii)  the only non-zero covariance is between $X'_{i,1} (0)$ and $X'_{i,1} (t)$.
 
 \item  After dividing  $X'_{i,1} (0)$ and $X'_{i,1} (t)$ by $\sigma_{max}(r)$  to obtain $   \tilde X'_{i,1} (0),\tilde X'_{i,1} (t) $
  the variance becomes
$$\tilde\sigma_i(r):=\frac{\sigma_i(r)}{\sigma_{\max}(r)}.$$
The domain for these variance, when $i$ varies  is 
$$K'_{2}:=\{\tilde\sigma(r) \in \R^d  \,:\,0\le\tilde\sigma_i(r)\le1,\,\mbox{ at least one }\tilde\sigma_i(r)=1\}.
$$
That is  a compact set.
 \item The domain for the covariance between  $   \tilde X'_{i,1} (0)$ and $ \tilde X'_{i,1} (t) $ 
 is given by Cauchy-Schwarz inequality : 
$$
\Cov\big(  \tilde X'_{i,1} (0),  \tilde X'_{i,1} (t) \big)  \leq \tilde\sigma_i(r),$$
that defines another compact set. 

\item  The other variables are independent  between them and independent of the variables above. Their variance are fixed.
\end{itemize} \medskip
It remains to prove that for every element of $K_1$  and $K_2$ the Gaussian distribution satisfies 
$$\mathbb P\{\det (X)=0\}=0\mbox{ and }  \mathbb P\{\det (Y)=0\}=0,$$
where $X,Y$  is a representation of the  conditional distribution  of $\tilde X'(0), \tilde X'(t)$. 
It is sufficient to study the case of $\det(X)$. Recall that  we have proved above that all the coordinates of $X$ are independent. The only difficulty is that  the variance  of the first column may vanish.

Let us consider the $d\times(d-1)$ matrix $X'_{:,-1}$ that consists of columns $2,\ldots,d$ of $X$. Because all the entries  of $X_{:,-1}$ are independent  they span a subspace of dimension $d-1$ a.s. 

%

Then the rank of $X'_{:,-1}$ is almost surely $(d-1)$ or $\mbox{Im}(X'_{:,-1})$ is a $(d-1)$ space. The distribution of the random matrix $ Y:=X'_{:,-1}$ has a density that can be written 

$$
f_Y[dY]=(Const) e^{-\frac12\mbox{Trace}(Y^\top\Sigma^{-1}Y)}[dY],$$ where as before we have denoted $\Sigma$ the covariance matrix of each column vector, which is diagonal. This density function is translated on the Grassmannian giving a bounded density  with respect to the  Haar measure.


Recall that  we have proved above that all the coordinates of $X$ are independent.
Let $X_{:,1}$ be the first column of $X$.
Conditioning on $X_{:,-1}$, by independence, the distribution of $X_{:,1}$ remains unchanged. A representation for this random variable is
$$X_{:,1}=\mu+\xi,\mbox{ with }\mu=\E(X_{:,1}),$$ where $\xi$ (because some $\tilde \sigma_i$ can vanish) has an absolute continuous distribution on the space 
$$S_I=(\xi_i=0,\mbox{ for } i\in I),\mbox{ with } I=\{i:\, \tilde\sigma_i=0\}.
$$
But  since  at least  one $\tilde \sigma_i=1$, we have $S_I\neq\{1,\ldots,d\}$.

Because of its  absolute continuity, almost surely, $ \xi$ can not be included in a given subspace $E$ that does not contain $S_I$.  In conclusion, given its absolutely continuous distribution over the Grassmannian, with probability one, $\mbox{Im}(X_{:,-1})$ cannot contain any  fixed affine space.


As a consequence we can apply Lemma  \ref{Chichi} (b) for getting the inequality in the other direction.

It remains to give a lower bound to the density.
$$p_{X(0),X(t)}(u,u)=\prod_{i=1}^d\frac1{2\pi}\frac1{\sqrt{1-\rho^2_i(r^2)}}e^{-\frac{u^2}{1-\rho_i(r^2)}}.$$
Since $\rho_i(r^2)\to1$ as $r\to0$ the term $e^{-\frac{u^2}{1-\rho_i(r^2)}}$ is lower bounded. Then is suffices to use (\ref{density}).

 \end{proof}

 %

 \subsection{Critical points} \label{criti}
Let $ Y : Y(t)$  be a random field  from $\R^d \to \R$. Critical points  points  of $ Y$ are in fact  zeros of $X= X'(t)$. 
  Strictly speaking  this process  does not satisfies  the hypotheses  of Theorem \ref{t:jma2} because $X'(t)$ is the Hessian of  $Y(t)$  so it is symmetric  and its distribution  in $\R^{d^2}$ is not N.D. But the result holds true  with a very similar proof mutatis mutandis.
  
  \begin{theorem}\label{t:jma2} Suppose    that 
  \begin{itemize} 
  \item$Y$ is Gaussian stationary, centred  and  has  $ C^2$ sample paths. 
  \item  $  Y'(t)$  is N.D., $Y''(t)$ has a non degenerated  distribution in the space  of symmetric matrices of  dimension $d\times d$.
  \end{itemize}
  Define 
  $$
\bar S^2_{max} (r) := \max_{i =1,\ldots,d}\max_{\lambda\in \mathbb S^{d-1}}    \Var\big(Y''_{i\lambda}\big |\, Y'(0), Y''(\lambda r)\big),\\
$$
   if $$
  \int \frac{  \bar S^2_{max} (r)}{r} dr  \mbox{ converges at } 0,
 $$
 then for all compact $S \subset  \R^d $ :   the second moment of the number of critical points  of $Y$  included in $S$ is finite. 
 \end{theorem}

 \section{Random fields from $\R^D$ to $\R^d$, $d<D$} \label{Dd}
 
  In this section  we study  the level sets  of a random field  $\R^D$ to $\R^d$. Of course the case
 $d>D$   has no interest  because  almost surely  the  level set is empty. The  case $d=D$ has been considered in the preceding sections so we assume    $d<D$.  The result  presented here is, in some sense, a by-product of Theorem  \ref{t:jma}, but by its simplicity  it is the most surprising result and one of the main results  of this paper.

  \begin{theorem}\label{t:jma3}
   Let $X:X(t)$ a stationary  random field $\R^D$ to $\R^d$, $d<D$ with $C^1$ paths. By the implicit function   theorem, a.s.   for every $u$  the level set $C_u$ is a manifold and its $D-d$ dimensional measure $\sigma_{D-d}$ is well defined. Let $C(u,K)$ be   the restriction of $C_u$ to a compact set $K \subset \R^D$. Assume that the distributions of $X(t)$ and $X'(t)$ are N.D. \\
   Then for every $u$ and $K$ 
   \begin{equation}\label{dD}
   \E \big(\sigma^2_{D-d}C(u,K)) < + \infty.
   \end{equation}
  \end{theorem}

\begin{proof}

The Kac-Rice formula reads
\begin{multline*} \E \big(\sigma^2_{D-d}C(u,K)) \\
= \int_{K^2}
\E_\mathcal C\big(\big(\det (X'(s)X'(s)^\top)\det (X'(t)X'(t)^\top)\big)^{\frac12}|\,|\,\big)\\ 
p_{X(s),X(t)} (u,u)  ds dt,
\end{multline*}
 where again $\E_\mathcal C$ denotes the  expectation conditional to $\mathcal C = \{X(0)=X(t)=u\}$.  
Using the arguments in the proof of Theorem \ref{t:jma},  we have
$$
p_{X(0),X(t)} (u,u)   \leq (Const)  \|t\|^{-d} .
$$
By Cauchy-Schwarz  inequality and symmetry
$$A(t,u):
 =\E_\mathcal C\left(\big(\det (X'(0)X'(0)^\top)\det (X'(t)X'(t)^\top)\big)^{\frac12}| \right)$$
 $$
  \leq  \E_\mathcal C\left(\big(\det (X'(0)X'(0)^\top)|\big)\right).
 $$
 Using \eqref{zaza} we have to bound 
 $$
  \E_\mathcal C \prod_{i=1} ^d \|  \nabla X_i(0)\|^2.
 $$
 Now  we borrow results from the proof of Theorem  \ref{t:jma}, to get that for every $i$:
  \begin{align*}
  \E_\mathcal C(  X'_{i,1}(0))  &  \to 0 \\
  \E_\mathcal C( X'_{i,j}(0))  &  \mbox{ is bounded } \quad j \neq 1,
  \end{align*}
  Because of the contracting property  of the conditional expectation
   $ \Var_\mathcal C( X'_{i,j}(0))$ is bounded. So, it follows directly that $\mathcal A(t,u)$ is upper-bounded.
  The integrability  of $t^{-d}$ in $\R^D$ gives the result. 
 \end{proof}
 

\section{Appendix}
 \begin{proof}[Proof of lemma \ref{integrales}]
Let us consider the integral
$$ \int_0^\delta\frac{\sigma^2(\tau)}{\tau}d\tau.$$
For $\tau$ small enough
$$\frac{\sigma^2(\tau)}{\tau}\sim(\frac1{\lambda_2})^{\frac32}\frac{\lambda_2(1-r^2(\tau))-(r'(\tau))^2}{\tau^3}.$$
Then integrating by parts\\
\\
$\displaystyle \int_0^{\delta}\frac{\lambda_2(1-r^2(\tau))-(r'(\tau))^2}{\tau^3}d\tau$
$$=\frac{\lambda_2(1-r^2(\delta))-(r'(\delta))^2}{2\delta^2}+\int_0^\delta \frac{r'(\tau)}\tau (\frac{-\lambda_2r(\tau)-r''(\tau)}\tau) d\tau.$$Hence we need to consider the second term that is equal to
$$-\lambda_2 \int_0^\delta\frac{r'(\tau)}\tau(\frac{r(\tau)-1}{\tau})d\tau-\int_0^\delta\frac{r'(\tau)}\tau(\frac{r''(\tau)+\lambda_2}{\tau})d\tau,$$as the first term is evidently convergent, the above sum 
is convergent if  and only if 
$$\int_0^{\delta}\frac{r''(\tau)+\lambda_2}{\tau}d\tau<\infty.$$
\end{proof}

\begin{proof}[Proof of lemma \ref{resultados}] For short we will write $r, r', r''$ instead of $r(\tau), r'(\tau), r''(\tau)$. Items $(a)-(d)$ are easy consequences of regression formulas (see  \cite{AW} page 100 for example). 

To prove (e)  we study first the behavior of 
$\frac{r'}{\sigma(\tau)}$ near to zero. We need consider two  cases.

The first one is when  the fourth spectral moment $ \lambda_4$ is finite: we have $r(t) = 1-\lambda_2 t^2/2 + \lambda_4 t^4/(4!) + o(t^4)$ . By using a Taylor expansion  of fourth order  on the numerator and the denominator  of the fraction  
$(\frac{r' u}{(1+r)\sigma(\tau)})^2,$ we obtain 
$$(\frac{r' u}{(1+r)\sigma(\tau)})^2\to\frac{\lambda_2^2u^2}{\lambda_4-\lambda_2^2}\le (Const)\,u^2, \quad \mbox{ giving  {\em (e)}}.$$ \medskip

Consider now the second case : $\lambda_4=+\infty$. \\
Given that 
$r''(\tau)-r''(0)=\int_0^\infty\frac{1-\cos(\tau\lambda)}{\tau^2 \lambda^2/2}d\mu(\lambda)$, we have by Fatou's lemma
\begin{align} \label{e:jma}
\liminf_{\tau\to0}\frac{r''(\tau)-r''(0))}{\tau^2}& \ge\int_0^{+\infty}\liminf_{\tau\to0}\frac{1-\cos(\tau\lambda)}{\tau^2 \lambda^2/2} \lambda^4d\mu(\lambda)   \notag \\
&=\int_0^\infty\lambda^4d\mu(\lambda)=+\infty.
\end{align}
Since $1+r$ tends to $2$, it suffices to study the behaviour of  
$$\frac{r'^2}{\sigma^2(\tau)}\simeq \frac{\lambda_2^3}{\frac{\lambda_2(1-r^2)-r'^2}{\tau^4}}\, ,$$
Note that 
$\lambda_2(1-r^2(\tau))-r'^2(\tau)= 2\,\lambda_2(1-r(\tau))-r'^2(\tau)+O(\tau^4)$. Furthermore by using the l'Hospital rule
$$\lim_{\tau\to0}\frac{2\,\lambda_2(1-r(\tau))-r'^2(\tau)}{\tau^4}= \lim_{\tau\to0} \left({\frac{-r'(\tau)}{2\tau}}\right)\left({\frac{r"(\tau)-r"(0))}{\tau^2}}\right)=+\infty,$$
 because of \eqref{e:jma} and since we know that $\frac{-r'(\tau)}{2\tau}\to \frac{\lambda_2}{2}$.
Thus $\frac{r'(\tau)}{\sigma(\tau)}\to 0$. These two results imply that  {\em(e)} holds.
\end{proof}

\begin{proof}[Proof of the lemma \ref{acotacion}] We can write
$$Y_2-m_2=\rho (Y_1-m_1)+\sqrt{1-\rho^2}Z_1,
$$
where  $Z_1$ is a standard Gaussian independent of $Y_1$.
Thus\\
\\
$\displaystyle Y_1Y_2=(m_1m_2+\rho)+(m_2+\rho m_1)(Y_1-m_1)$
$$+m_1\sqrt{1-\rho^2}Z_1+\rho ((Y_1-m_1)^2-1)+\sqrt{1-\rho^2}(Y_1-m_1)Z_1.
$$
This formula yields that $\E|Y_1Y_2|$ is a continuous function of $(m_1,m_2,\rho)$ and by compactness is upper bounded. 

In the other direction, setting $Y_i= m_i + \bar Y_i$, $i=1,2$  we have 
$$ 
Y_1 Y_2 = [m_1 m_2 ]+[ m_1 \bar Y_2 +   m_2 \bar Y_1] + [ \bar Y_1\bar Y_2  ],
$$
Where the three brackets  are in different It\^o-Wiener chaos and thus  the joint variance is the sum of the variance of each term. This implies that 
$$
\Var(Y_1 Y_2)  \geq \Var( \bar Y_1 \bar Y_2) = \E ( \bar Y_1^2 \bar Y_2^2)= 
 \E \Big( \big(H_2 (\bar Y_1) +1\big) \big(H_2 (\bar Y_2) +1\big) \Big) = 1 +2 \rho^2,
$$  
the function  $H_2(y)  = y^2-1 $ is the second Hermite polynomial and also we have used  the Mehler formula:
$\E( H_i(\bar Y_1) H_j(\bar Y_2 ) = \delta_{ij} \rho ^i i!$. Note that $H_0(y)  = 1 $. 

As a consequence 
$$\Var(Y_1Y_2)\ge 1.$$ and  $Y_1Y_2$ is never a.s. constant. Thus $\E|Y_1Y_2|>0$ and by compactness is lower bounded.
\end{proof}

 \begin{proof} [Proof of lemma \ref{l:gos}]
 Note that for $n$  large enough,  the distribution of $ Z_n$ is  N.D. As a consequence
 \begin{align*}
 \E(T|\,Z_n=z)&=\Cov(T,Z_n)(\Var(Z_n))^{-1}z \\
 \Var(T|Z_n) &= \Var(T)- \Cov(T,Z_n)(\Var(Z_n))^{-1}\Cov(Z_n,T)
 \end{align*}
 but
 $$\Var(Z_n)\to\Var(Z),\mbox{ N.D.}$$This implies 
 $(      \Var(Z_n)    )^{-1}\to(\Var(Z))^{-1}).$ The rest is plain.
 \end{proof}

 \begin{proof}[Proof of lemma \ref{Chichi}]

Let $\Sigma$ be the variance-covariance matrix  of $Z$  and let $\mu $ be  its expectation. Both vary  in a compact sets $K_1,K_2$.  Let $\Sigma^\frac 1 2$ be the square root of $\Sigma$  defined in the spectral way. Using operator norm \cite{Far:Nik}, it is easy to prove  that $\Sigma^\frac 1 2$  is a (uniformly) continuous  function of $\Sigma$. The random vector $Z$ admits the following représentation
$$
Z = \mu  + \Sigma^\frac 1 2 \xi \quad \xi \sim  N(0,I_{2d^2}).
$$ 
 The function $\det(X) \det(Y)$, as a polynomial, is a continuous function  of $Z$  and by consequence 
 $ \E\big(|\det(X) \det(Y)|\big)$ is a continuous function of $ \mu,\Sigma$.  The first assertion follows by compactness. \medskip

 In the other direction  we have by additivity
 $$\P\{ \det(X)=0\} =0, \quad \P\{ \det(Y)=0\} =0.
 $$
  This implies that 
  $$
  \E(|\det(X)\det(Y)|) >0.
  $$
  Again the second  inequality is obtained by compactness.
 \end{proof}


\begin{bibdiv}
\begin{biblist}

\end{biblist}
\end{bibdiv}


\begin{thebibliography}{9}

\bibitem{Ad1} R. Adler.
{\em The geometry of random fields.} {John Wiley \& Sons, Ltd., Chichester}, (1981).
\bibitem{Ad:Ha} R. Adler \& A.M. Hasofer.
{\em Level crossings for random fields.}{ Ann. Probability} 4  no. 1, 1-12 (1976).
\bibitem{Ad:Tay} R. Adler \& J. Taylor.
{\em Random fields and geometry}. {Springer Monographs in Mathematics. Springer, New York.} (2007).
\bibitem{Au:Ben} A. Auffinger \& G. Ben Arous. 
{\em Complexity of random smooth functions on the high-dimensional sphere.} {Ann. Probab. }41  No. 6, 4214-4247 (2013).
\bibitem{AW} J-M Aza\"{\i}s \& M. Wschebor. 
{\em Level sets and extrema of random processes and fields}. {John Wiley \& Sons, Inc., Hoboken, NJ.} (2009).
\bibitem{berry1}M. V. Berry, M.V.
{\em  Statistics of nodal lines and points in chaotic quantum billiards:
perimeter corrections, fluctuations, curvature.} {Journal of Physics} A, 35, 3025-3038 (2002).
\bibitem{C:L}H. Cramer \& M. R. Leadbetter.
{\em Stationary and related stochastic processes. Sample function properties and their applications.}{ John Wiley \& Sons, Inc.} (1967).
\bibitem{estr:four} A. Estrade \& J. Fournier.
{\em Number of critical points of a Gaussian random field: condition for a finite variance.} {Statist. Probab. Lett.} 118  94-99 (2016).
\bibitem{Far:Nik} Yu. B Farforovskaya \& L. Nikolskaya.
{\em Modulus of continuity of operator functions} {St. Petersburg Math. J.} Vol. 20, No. 3, Pages 493-506 (2009).
\bibitem{Ge} D. Geman.
{\em Occupation times for smooth stationary processes.} {Ann. Probability }1 No. 1, 131-137  (1973). 
\bibitem{Kac:Kac} M. Kac.
{\em On the average number of real roots of a random algebraic equation. } { Bull. Amer. Math. Soc.} 49 282-332 (1944).
\bibitem{Kratz:Leon}M. Kratz \& J.R. León.
{\em On the second moment of the number of crossings by a stationary Gaussian process.}{ Ann. Probab. } 34, No. 4, 1601-1607 (2006).
\bibitem{KL} M. Kratz \& J.R. León.
{\em Central Limit Theorems for Level Functionals of Stationary Gaussian Processes and Fields.} {Journal of Theoretical Probability.}  14, No. 3 639-672 (2001).
\bibitem{let1}T.  Letendre.
{\em Expected volume and Euler characteristic of random submanifolds.} {J. Funct. Anal. } 270 no. 8, 3047-3110 (2016).
\bibitem{Lindgren} G. Lindgren \& F. Lindgren.
{\em Stochastic asymmetry properties of 3D Gauss-Lagrange ocean waves with directional spreading.}
{Stochastic Models,} 27 490-520 (2011).
\bibitem{Peccati}D. Marinucci, G. Peccati, M. Rossi, I. Wigman.
{\em Non-universality of nodal length distribution for arithmetic random waves.}{ Geom. Funct. Anal. } 26  No. 3, 926-960 (2016).
\bibitem{Rice1} S. O. Rice.
{\em Mathematical analysis of random noise I.} { Bell System Tech. J.} 23 314-320 (1943).
\bibitem{Rice2} S.O. Rice.
{\em Mathematical analysis of random noise II.} { Bell System Tech. J.} 24 46-156  (1945).
\bibitem{W1} M. Wschebor.
{\em  M\'esure g\'eom\'etrique des ensembles de niveau.}{ LNM 1147 Springer-Verlag }(1985).
\end{thebibliography}
\end{document}